\documentclass[12pt,a4paper,oneside,reqno]{amsproc}
\usepackage{amsmath} 
\usepackage{amssymb,amsfonts,mathrsfs}
\usepackage[colorlinks=true,linkcolor=blue,citecolor=blue]{hyperref}
\usepackage[driver=pdftex,lmargin=2.5cm,rmargin=2.5cm,heightrounded=true,centering]{geometry}
\usepackage{math50}
\usepackage{local}
\usepackage{epsfig}  
\usepackage{graphicx}  
\usepackage{xy}
\xyoption{all}

\begin{document}

\title[Neumann problem for fractional Schr\"odinger equations]
{Singularly perturbed Neumann problem for fractional Schr\"odinger equations}
\author{Guoyuan Chen}
\address{School of Data Sciences, Zhejiang University of Finance \& Economics, Hangzhou 310018, Zhejiang, P. R. China}
\email{gychen@zufe.edu.cn}
\newcommand{\optional}[1]{\relax}
\setcounter{secnumdepth}{3}
\setcounter{section}{0} \setcounter{equation}{0}
\numberwithin{equation}{section}

\date{\today}
\begin{abstract}
This paper is concerned with a Neumann type problem for singularly perturbed fractional nonlinear Schr\"odinger equations with subcritical exponent. For some smooth bounded domain $\Omega\subset \mathbf R^n$, our boundary condition is given by
\begin{equation*}
\int_{\Omega}\frac{u(x)-u(y)}{|x-y|^{n+2s}}dy=0\quad\mbox{for }x\in \mathbf R^n\setminus\bar\Omega.
\end{equation*}
We establish existence of nonnegative small energy solutions, and also investigate the integrability of the solutions on $\mathbf R^n$.
\end{abstract}
\maketitle

\section{Introduction}

The main purpose of this paper is to investigate a singularly perturbed Neumann type problem for fractional Schr\"odinger equations. Precisely, given a smooth bounded domain $\Omega\subset \mathbf R^n$, we consider the following problem
\begin{equation}\label{e:np}
\left\{\begin{array}{ll}
  \varepsilon^{2s}(-\Delta)^su+u=|u|^{p-1}u & \mbox{ in }\Omega, \\
  \mathcal N_su=0 & \mbox{ on }\mathbf R^n\setminus\bar\Omega.
\end{array}\right.
\end{equation}
Here $\varepsilon>0$, $s\in (0,1)$, $n\ge 2$, $p\in (1,\frac{n+2s}{n-2s})$, and
\begin{equation}\label{e:ns}
\mathcal N_su(x)=C_{n,s}\int_{\Omega}\frac{u(x)-u(y)}{|x-y|^{n+2s}}dy,\quad x\in \mathbf R^n\setminus\bar\Omega,
\end{equation}
where $C_{n,s}$ is the normalization constant in the definition of fractional Laplacian
\begin{equation*}
(-\Delta)^su(x)=C_{n,s}{\rm P.V.}\int_{\mathbf R^n}\frac{u(x)-u(y)}{|x-y|^{n+2s}}dy.
\end{equation*}

 This type of boundary problem for fractional Laplacian was introduced by Dipierro, Ros-Oton and Valdinoci in \cite{DRV:NPNBC}. It corresponds to a jump process as follows: If a particle has gone to $x\in \mathbf R^n\setminus \bar\Omega$, then it may come back to any point $y\in \Omega$, the probability of jumping from $x$ to $y$ being proportional to $|x-y|^{-n-2s}$. From mathematical point of view, such kind of boundary conditions generalize the classical Neumann conditions for elliptic (or parabolic) differential equations. That is, if $s\to 1$, then $\mathcal N_s u=0$ becomes the classical Neumann condition. For more details, see \cite{DRV:NPNBC}. Also in \cite{Du&Gu&Le&Zh12,Du&Gu&Le&Zh13}, Du-Gunzburger-Lehoucq-Zhou introduced volume constraints for a general class of nonlocal diffusion problems on bounded domain in $\mathbf R^n$ via a nonlocal vector calculus. If we rewrite (\ref{e:ns}) by a nonlocal vector calculus for fractional Laplacian, then $\mathcal N_s u=0$ (with some modifications) can be considered as a special case of the volume constraints defined by \cite{Du&Gu&Le&Zh12,Du&Gu&Le&Zh13}.

 Other types of Neumann problems for fractional Laplacian (or other nonlocal operators) were investigated in many works \cite{Bogdan&Burdzy&Chen03,Chen&Kim02,Bar&Cha&Geo&Jak14,Bar&Geo&Jak14,Cor&Elg&Ros&Wol07,Cor&Elg&Ros&Wol08,Grubb14}.
 All these conditions also have probabilistic interpretations and recover the classical Neumann problem as a limit case. A comparison between these models and ours can be found in \cite[Section 7]{DRV:NPNBC}.

The singularly perturbed Neumann problem for classical nonlinear Schr\"odinger equations with subcritical exponent is as follows:
\begin{equation}\label{e:classical}
\left\{\begin{array}{ll}
         -\varepsilon^2\Delta u+u=u^p & \mbox{ in }\Omega, \\
         \partial_{\nu}u=0 & \mbox{ on } \partial\Omega,\\
         u>0 & \mbox{ in }\Omega,
       \end{array}
\right.
\end{equation}
where $1<p<\frac{n+2}{n-2}$ for $n\ge 3$ and $p>1$ for $n=2$, and $\nu$ is the unit outward normal on $\partial \Omega$. There is a great deal of works on this problem. We only restrict ourselves to cite a few papers, referring to the bibliography for further references.
The pioneer works by Lin, Ni and Tagaki \cite{Ni&Takagi86,Lin&Ni&Takagi88,Ni&Takagi91,Ni&Takagi93} proved the existence of single-peak spike layer solution $u_{\varepsilon}$ to (\ref{e:classical}). After that many interesting results concerning multi-peak spike-layer solutions to (\ref{e:classical}) have been obtained (\cite{Gui&Wei99,Gui&Wei00,Gui&Wei&Winter00}). Note that a spike-layer solution has its energy or mass concentrating near isolated points (a zero-dimensional set) in $\bar\Omega$. Similarly, there exist solutions to (\ref{e:classical}) with $k$-dimensional ($1\le k\le n-1$) concentration set (\cite{Malchiodi&Montenegro02,Malchiodi&Montenegro04,AMN:SPEE,Amb&Mal&Ni04,DP&Kow&Wei07}). We refer to \cite{Ni11} for more other results and references. Moreover, concentrating standing waves with a critical frequency for Schr\"odinger equations were obtained by \cite{BW:SWCF,byeon2006spherical}.

We should also mention that concentration phenomenon for fractional Schr\"odinger equations has been extensively studied recently. On the total space $\mathbf R^n$, the existence and multiplicity of spike layer solutions under various conditions were obtained by \cite{DDW:CSWFSE,CZ:CPFSE,chen2015multiple,FaMaVa14}. On a bounded domain in $\mathbf R^n$, singularly perturbed Dirichlet problem was investigated by \cite{DaDPDiVa14}. Moreover, under classical Neumann condition, an existence result of spike solutions to Schr\"odinger equations involving half Laplacian (see Equation (\ref{e:classical2}) below) was proved by \cite{stinga2015fractional}.

We are now in a position to formulate our main results and give the idea of the proofs.
Our problem (\ref{e:np}) has a variational structure. More precisely, let
\begin{eqnarray}\label{e:inner-product}
\langle u,v\rangle_{H^s_{\varepsilon,\Omega}}=\frac{C_{n,s}\varepsilon^{2s}}{2}\int_{\mathbf R^{2n}\setminus(\Omega^c)^2}\frac{(u(x)-u(y))(v(x)-v(y))}{|x-y|^{n+2s}}dxdy+\int_{\Omega}uvdx,
\end{eqnarray}
where $\Omega^c:=\mathbf R^n\setminus \Omega$.
Then define the space
\begin{eqnarray*}
H^s_{\varepsilon,\Omega}=\{u:\mathbf R^n\to \mathbf R \mbox{ measurable and } \langle u,u\rangle_{H^s_{\varepsilon,\Omega}}<\infty\}.
\end{eqnarray*}
It is a Hilbert space with the norm $\|\cdot\|_{H^s_{\varepsilon,\Omega}}=\langle \cdot,\cdot\rangle_{H^s_{\varepsilon,\Omega}}^{\frac{1}{2}}$.
It follows that weak solutions to the problem (\ref{e:np}) are critical points of the following functional
\begin{eqnarray}\label{e:functional-i}
I_{\varepsilon}(u)=\frac{C_{n,s}\varepsilon^{2s}}{4}\int_{\mathbf R^{2n}\setminus(\Omega^c)^2}\frac{|u(x)-u(y)|^2}{|x-y|^{n+2s}}dxdy+\frac{1}{2}\int_{\Omega}u^2dx-
\frac{1}{p+1}\int_{\Omega}|u|^{p+1}dx.
\end{eqnarray}
We obtain the following existence result.
\begin{theorem}\label{t:np}
If $\varepsilon$ is sufficiently small, then there exists a nonnegative weak solution $u_{\varepsilon}$ to (\ref{e:np}). Moreover, $u_{\varepsilon}$ satisfies
\begin{equation*}
0<I_{\varepsilon}(u_{\varepsilon})\le C_1\varepsilon^n.
\end{equation*}
 Consequently, $u_{\varepsilon}$ is a nonconstant solution, and,
\begin{equation}\label{e:solution-norm}
\|u_{\varepsilon}\|_{H^s_{\varepsilon,\Omega}}\le C_2\varepsilon^{\frac{n}{2}}.
\end{equation}
Here $C_1, C_2$ are two positive constants depending only on $n,s,p,\Omega$.
\end{theorem}

\begin{remark}
The definition of weak solution is given by integrals (see Definition \ref{e:weak-solutions} below). Using nonlocal integration by parts formulas, we see that this definition is similar to the classical case (see Remark \ref{r:weak-s} below).
\end{remark}
\begin{remark}
Such kind of existence results for classical Neumann problem (\ref{e:classical}) was obtained by Lin, Ni and Takagi \cite{Lin&Ni&Takagi88}. Recently, in \cite{stinga2015fractional}, Stinga and Volzone recovered the results (including concentration and regularity issues) in \cite{Lin&Ni&Takagi88} for a fractional semilinear Neumann problem as follows:
\begin{equation}\label{e:classical2}
\left\{\begin{array}{ll}
         \varepsilon(-\Delta)^{\frac{1}{2}} u+u=u^p & \mbox{ in }\Omega, \\
         \partial_{\nu}u=0 & \mbox{ on } \partial\Omega,\\
         u>0 & \mbox{ in }\Omega,
       \end{array}
\right.
\end{equation}
where $p\in (1,\frac{n+1}{n-1})$, $\nu$ is the outer unit normal to $\partial \Omega$. Note that both of the boundary conditions in (\ref{e:classical}) and (\ref{e:classical2}) are classical.
\end{remark}

The proof of Theorem \ref{t:np} relies on critical point theory. More precisely, the functional $I_{\varepsilon}$ has mountain pass structure. The key point is to construct an appropriate function $\phi\in H^s_{\varepsilon,\Omega}$ such that, for some $t_0>0$, it holds that  $I_{\varepsilon}(t_0\phi)\le 0$ and $0<\sup_{t\in[0,t_0]}I_{\varepsilon}(t\phi)\le C\varepsilon^n$ ($C$ is a constant depending on $n,s,\Omega$). As compared with the classical case, verifying the necessary properties of $\phi$ becomes more involved because of the fractional Laplacian. See Section \ref{s:mountain-pass} below.

Moreover, we investigate the integrability properties of the solutions to problem (\ref{e:np}). We have the following theorem.
\begin{theorem}\label{t:integrability}
Let $u\in H^s_{\varepsilon,\Omega}$. Then it holds that
\begin{enumerate}
  \item[(1)] $u\in L^2_{{\rm loc}}(\mathbf R^n)$,
  \item[(2)] if $\mathcal N_s(u)=0$, then
\begin{equation}\label{e:intgrality-sil}
\int_{\mathbf R^n}\frac{|u(x)|}{1+|x|^{n+2s}}dx<+\infty.
\end{equation}
\end{enumerate}
\end{theorem}
\begin{remark}
The second conclusion of this theorem implies that the nonnegative weak solution $u_{\varepsilon}$ obtained in Theorem \ref{t:np} is $L_s$ integrable in the sense of Silvestre \cite{Silvestre07}. We should note that if $\mathcal N_s(u)=0$, then $\lim_{|x|\to \infty}u(x)=\frac{1}{|\Omega|}\int_{\Omega}u(x)dx$ (\cite[Proposition 3.13]{DRV:NPNBC}). Therefore, in general, we can not expect to prove that $u_{\varepsilon}$ is integrable on the total space $\mathbf R^n$. In particular, $u_{\varepsilon}$ does not belong to the $s$-th order Sobolev space $H^s(\mathbf R^n)$. Hence it is too weak to obtain more regularity properties of $u_{\varepsilon}$ (see for example \cite{Silvestre07,Cabre&Sire12,jin2014fractional}).
\end{remark}

To prove Theorem \ref{t:integrability}, we need a detailed analysis of some singular integrals of the following form
\begin{equation*}
u[A,B]:=\int_{A}\int_{B}\frac{|u(x)-u(y)|^2}{|x-y|^{n+2s}}dydx,
\end{equation*}
where $A$ and $B$ are two measurable set in $\mathbf R^n$. Such kind of integral is also important in nonlocal minimal surfaces (see e.g. \cite{Caf&Roq&Sav10,Savin&Valdinoci11}). Since $u\in H^s_{\varepsilon,\Omega}$, $u(\Omega,\Omega)$ and $u(\Omega,\Omega^c)$ are finite. Then we can prove the local integrability of $u$ by choosing appropriate balls in $\mathbf R^n$. See Section \ref{s:properties} below.

\subsection*{Acknowledgements}
This work was supported in part by the National Natural Science Foundation of China (NSFC) (No. 11401521).


\section{Variational structure}\label{s:variational-structure}

The Neumann problem (\ref{e:np}) is variational.

Define the space
\begin{eqnarray*}
H^s_{\varepsilon,\Omega}=\{u:\mathbf R^n\to \mathbf R \mbox{ measurable and } \|u\|_{H^s_{\varepsilon,\Omega}}<\infty\}.
\end{eqnarray*}
Here $\|\cdot\|_{H^s_{\varepsilon,\Omega}}$ is given by (\ref{e:inner-product}).
We should note that
constant functions $v(x)\equiv c$ on $\mathbf R^n$ are contained in $H^s_{\varepsilon,\Omega}$.

\begin{remark}\label{r:sobolev-spaces}
Let $H^s(\Omega)$ be the $s$-th Sobolev space in $\Omega$ with the norm given by
\begin{eqnarray*}
\|h\|_{H^s(\Omega)}^2&=&\int_{\Omega}\int_{\Omega}\frac{|h(x)-h(y)|^2}{|x-y|^{n+2s}}dxdy+\int_{\Omega}h^2dx.
\end{eqnarray*}
(See, for example, \cite{Ad:SS}, \cite{DiNezza&Palatucci&Valdinoci12}.) Therefore, if $u\in H^s_{\varepsilon,\Omega}$, then $u|_{\Omega}$ is in $H^s(\Omega)$. We define $H^s(\mathbf R^n)$ to be the usual $s$-th Sobolev space in $\mathbf R^n$ with the norm
\begin{eqnarray*}
\|h\|_{H^s(\mathbf R^n)}^2&=&\int_{\mathbf R^n}\int_{\mathbf R^n}\frac{|h(x)-h(y)|^2}{|x-y|^{n+2s}}dxdy+\int_{\mathbf R^n}h^2dx.
\end{eqnarray*}
From (\ref{e:inner-product}), it holds that $H^s(\mathbf R^n)\subset H^s_{\varepsilon,\Omega}$.
\end{remark}
\begin{lemma}
$H^s_{\varepsilon,\Omega}$ is a Hilbert space with inner product given by (\ref{e:inner-product}).
\end{lemma}
\begin{proof}
This lemma is the case $g=0$ of Proposition 3.1 in \cite{DRV:NPNBC}. We omit details of the proof here.
\end{proof}

\begin{remark}
\emph{From Remark \ref{r:sobolev-spaces} and Sobolev embedding $H^s(\Omega)\hookrightarrow L^q(\Omega)$ ($q\in (1,\frac{2n}{n-2s})$), we have that, for all $u\in H^s_{\varepsilon,\Omega}$,
$$\int_{\Omega}|u|^{p+1}dx<+\infty.$$}
\end{remark}

\begin{dfn}
\emph{We say that $u\in H^s_{\varepsilon,\Omega}$ is a weak solution to (\ref{e:np}), if, for all $v\in H^s_{\varepsilon,\Omega}$, it holds that
\begin{equation}\label{e:weak-solutions}
\frac{C_{n,s}\varepsilon^{2s}}{2}\int_{\mathbf R^{2n}\setminus(\Omega^c)^2}\frac{(u(x)-u(y))(v(x)-v(y))}{|x-y|^{n+2s}}dxdy+\int_{\Omega}uvdx-\int_{\Omega}|u|^{p-1}uvdx=0.
\end{equation}}
\end{dfn}
\begin{remark}\label{r:weak-s}
A direct computation yields that
for all $u,\,v\in C^2\cap H^s_{\varepsilon,\Omega}$,
\begin{equation}\label{e:green-identity}
\frac{1}{2}\int_{\mathbf R^{2n}\setminus(\Omega^c)^2}\frac{(u(x)-u(y))(v(x)-v(y))}{|x-y|^{n+2s}}dxdy=\int_{\Omega}v(-\Delta)^su dx+\int_{\Omega^c}v\mathcal N_sudx.
\end{equation}
This is corresponding to the classical Green's first identity.
Thus
the definition of weak solution is the same as the classical case.
That is, from (\ref{e:green-identity}), (\ref{e:weak-solutions}) can formally become
\begin{eqnarray*}
\int_{\Omega}\left(\varepsilon^{2s}(-\Delta)^su+u-|u|^{p-1}u\right)vdx+\varepsilon^{2s}\int_{\Omega^c}v\mathcal N_sudx=0.
\end{eqnarray*}
\end{remark}

\begin{prop}
Any critical point of $I_{\varepsilon}$ (see (\ref{e:functional-i})) is a weak solution of problem (\ref{e:np}).
\end{prop}
\begin{proof}
For any $v\in H^s_{\varepsilon,\Omega}$, we have that
\begin{eqnarray*}
&&I_{\varepsilon}(u+tv)\\
&=&\frac{C_{n,s}\varepsilon^{2s}}{4}\int_{\mathbf R^{2n}\setminus(\Omega^c)^2}\frac{|(u+tv)(x)-(u+tv)(y)|^2}{|x-y|^{n+2s}}dxdy\\
&&\quad\quad\quad\quad\quad\quad\quad\quad+\int_{\Omega}(u+tv)^2dx-
\frac{1}{p+1}\int_{\Omega}|u+tv|^{p+1}dx\\
&=&I_{\varepsilon}(u)\\
&&+t\left(C_{n,s}\varepsilon^{2s}\int_{\mathbf R^{2n}\setminus(\Omega^c)^2}\frac{(u(x)-u(y))(v(x)-v(y))}{|x-y|^{n+2s}}dxdy+\int_{\Omega}uvdx-\int_{\Omega}|u|^{p-1}uvdx\right)\\
&&+t^2\left(\frac{C_{n,s}\varepsilon^{2s}}{4}\int_{\mathbf R^{2n}\setminus(\Omega^c)^2}\frac{|v(x)-v(y)|^2}{|x-y|^{n+2s}}dxdy+\int_{\Omega}v^2dx+p\int_{\Omega}|u+\theta_1tv|^{p-1}v^2dx\right),
\end{eqnarray*}
where $\theta_1\in (0,1)$. Thus
\begin{eqnarray}\label{e:derivative-0}
I'_{\varepsilon}(u)v&=&\lim_{t\to 0}\frac{I_{\varepsilon}(u+tv)-I_{\varepsilon}(u)}{t}\notag\\
&=&C_{n,s}\varepsilon^{2s}\int_{\mathbf R^{2n}\setminus(\Omega^c)^2}\frac{(u(x)-u(y))(v(x)-v(y))}{|x-y|^{n+2s}}dxdy+\int_{\Omega}uvdx-\int_{\Omega}|u|^{p-1}uvdx\notag\\
&=& \langle u,v\rangle_{H^s_{\varepsilon,\Omega}}-\int_{\Omega}|u|^{p-1}uvdx.
\end{eqnarray}
Therefore, if $u$ is a critical point of $I_{\varepsilon}$, then $u$ is a weak solution to (\ref{e:np}).
\end{proof}


\section{Proof of Theorem \ref{t:np}}\label{s:mountain-pass}

\begin{lemma}
$I_{\varepsilon}$ satisfies Palais-Smale condition.
\end{lemma}
\begin{proof}
Let $\{u_m\}\subset H^s_{\varepsilon,\Omega}$ be a Palais-Smale sequence such that $|I_{\varepsilon}(u_m)|\le d$, for all $m\in \mathbf N$, is bounded and $I_{\varepsilon}'(u_m)\to 0$.
Then
\begin{eqnarray}\label{e:udb}
d\ge \left|\frac{1}{2}\|u_m\|_{H^s_{\varepsilon,\Omega}}^2-\frac{1}{p+1}\int_{\Omega}|u_m|^{p+1}dx\right|.
\end{eqnarray}
Since $I_{\varepsilon}'(u_m)\to 0$, for any $\epsilon>0$, there is an $M=M(\epsilon)$ such that for all $m\ge M$,
\begin{eqnarray}\label{e:fdvb}
&&|I_{\varepsilon}'(u_m)v|\\
&=&\left|\left(\frac{C_{n,s}\varepsilon^{2s}}{2}\int_{\mathbf R^{2n}\setminus(\Omega^c)^2}
\frac{(u_m(x)-u_m(y))(v(x)-v(y))}{|x-y|^{n+2s}}dxdy+\int_{\Omega}u_m vdx\right)\right.\notag\\
&&\quad\quad\quad\quad\quad\quad\quad\quad\quad\quad\quad\quad\quad\quad\left.-\int_{\Omega}|u_m|^{p-1}u_mvdx\right|\notag\\
&=&\left|\langle u_m,v\rangle_{H^s_{\varepsilon,\Omega}}-\int_{\Omega}|u_m|^{p-1}u_mvdx\right|\le\epsilon\|v\|_{H^s_{\varepsilon,\Omega}},\notag
\end{eqnarray}
for all $v\in H^s_{\varepsilon,\Omega}$. If we choose $\epsilon=1$, $v=u_m$, then (\ref{e:fdvb}) yields
\begin{eqnarray}\label{e:umbi}
\left|\int_{\Omega}|u_m|^{p+1}dx\right|\le \|u_m\|_{H^s_{\varepsilon,\Omega}}^2+\|u_m\|_{H^s_{\varepsilon,\Omega}}.
\end{eqnarray}
By (\ref{e:umbi}) and (\ref{e:udb}), we have that
\begin{eqnarray*}
d\ge \left(\frac{1}{2}-\frac{1}{p+1}\right)\|u_m\|_{H^s_{\varepsilon,\Omega}}^2-\frac{1}{p+1}\|u_m\|_{H^s_{\varepsilon,\Omega}}.
\end{eqnarray*}
Therefore, $\{u_m\}$ is bounded in $H^s_{\varepsilon,\Omega}$.
Up to a subsequence, we assume that $u_m\rightharpoonup u$ in $H^s_{\varepsilon,\Omega}$. By Remark \ref{r:sobolev-spaces} and Sobolev embedding, $u_m\to u$ in $L^{p+1}(\Omega)$. So $|u_m|^{p-1}u_m\to |u|^{p-1}u$ in $L^{(p+1)/p}(\Omega)$. Equation (\ref{e:derivative-0}) yields that
\begin{eqnarray*}
\|u_m-u\|_{H^s_{\varepsilon,\Omega}}^2&=& \langle I_{\varepsilon}'(u_m)-I_{\varepsilon}'(u),u_m-u\rangle_{H^s_{\varepsilon,\Omega}}\\
&&+\int_{\Omega}(|u_m|^{p-1}u_m-|u|^{p-1}u)(u_m-u).
\end{eqnarray*}
Since $\{u_m\}$ is bounded, $u_m\rightharpoonup u$ in $H^s_{\varepsilon,\Omega}$ and $I_{\varepsilon}'(u_m)\to 0$, we have that
\begin{eqnarray*}
\langle I_{\varepsilon}'(u_m)-I_{\varepsilon}'(u),u_m-u\rangle_{H^s_{\varepsilon,\Omega}}\to 0, \quad\mbox{as } m\to \infty.
\end{eqnarray*}
By H\"older inequality, it holds that
\begin{eqnarray*}
&&\left|\int_{\Omega}\left(u_m|^{p-1}u_m-|u|^{p-1}u\right)(u_m-u)\right|\\
&\le& \left\||u_m|^{p-1}u_m-|u|^{p-1}u\right\|_{L^{(p+1)/p}(\Omega)}\|u_m-u\|_{L^{p+1}(\Omega)}\to 0,
\end{eqnarray*}
as $m\to \infty$. Therefore, we have that $\|u_m-u\|_{H^s_{\varepsilon,\Omega}}\to 0$ as $m\to \infty$.
It completes the proof.
\end{proof}

\begin{lemma}\label{l:positive}
There exists a $\rho>0$ such that $I_{\varepsilon}(u)>0$ if $0<\|u\|_{H^s_{\varepsilon,\Omega}}<\rho$ and $I_{\varepsilon}(u)\ge \beta>0$ if $\|u\|_{H^s_{\varepsilon,\Omega}}=\rho$.
\end{lemma}
\begin{proof}
By Remark \ref{r:sobolev-spaces} and Sobolev embedding,
\begin{eqnarray*}
\int_{\Omega}|u|^{p+1}dx\le \|u\|_{H^s_{\varepsilon,\Omega}}^{p+1}.
\end{eqnarray*}
Since $p>1$, the conclusion of this lemma holds.
\end{proof}

\begin{lemma}\label{l:energy-estimate}
For sufficiently small $\varepsilon>0$, there exist a nonconstant function $\phi\in H^s_{\varepsilon,\Omega}$ and positive constants $t_0$ such that $I_{\varepsilon}(t_0\phi)=0$ and $I_{\varepsilon}(t\phi)\le C\varepsilon^n$ if $t\in [0,t_0]$. Here $C$ is a constant depending on $n,s, \Omega$.
\end{lemma}

To prove this lemma, we construct a special function. Without loss of generality, we assume that $0\in \Omega$. Let $\varepsilon\in (0,1)$ is small enough so that $B_{2\varepsilon}\subset \Omega$. Define
\begin{equation*}
\phi(x)=\left\{\begin{array}{cc}
                 \varepsilon^{-n}(1-\varepsilon^{-1} |x|) & \mbox{if } |x|< \varepsilon, \\
                 0 & \mbox{if } |x|\ge \varepsilon .
               \end{array}
\right.
\end{equation*}

\begin{lemma}\label{l:phi-estimate}
For sufficiently small $\varepsilon$, $\phi\in H^s_{\varepsilon,\Omega}$. Precisely, we have that
\begin{equation}\label{e:phi-estimate}
\|\phi\|^2_{H^s_{\varepsilon,\Omega}}\le \frac{C}{\varepsilon^{n}},
\end{equation}
where $C$ is a positive constant depending on $n$, $s$, $p$ and $\Omega$.
\end{lemma}
\begin{proof}
A direct calculus yields
\begin{eqnarray}\label{e:phi-2}
\int_{\Omega}\phi^2(x)dx&=&\int_{B_{\varepsilon}}\frac{1}{\varepsilon^{2n}}\left(1-\frac{|x|}{\varepsilon}\right)^2dx\notag\\
&=&\frac{\omega_{n-1}}{\varepsilon^{2n}}\int_0^\varepsilon\left(1-\frac{r}{\varepsilon}\right)^2r^{n-1}dr=\left(\frac{2\omega_{n-1}}{n(n+1)(n+2)}\right)\frac{1}{\varepsilon^{n}}.
\end{eqnarray}
Here $\omega_{n-1}$ is the area of unit sphere in $\mathbf R^n$. Thus it remains us to estimate
\begin{eqnarray}\label{e:phi-integral}
\frac{C_{n,s}\varepsilon^{2s}}{2}\int_{\mathbf R^{2n}\setminus(\Omega^c)^2}\frac{|\phi(x)-\phi(y)|^2}{|x-y|^{n+2s}}dxdy.
\end{eqnarray}
Compute
\begin{eqnarray}\label{e:integrals-1}
&&\int_{\mathbf R^{2n}\setminus(\Omega^c)^2}\frac{|\phi(x)-\phi(y)|^2}{|x-y|^{n+2s}}dxdy\notag\\
&=&\int_{\Omega}\int_{\Omega}\frac{|\phi(x)-\phi(y)|^2}{|x-y|^{n+2s}}dxdy+\int_{\Omega^c}\int_{\Omega}\frac{|\phi(x)-\phi(y)|^2}{|x-y|^{n+2s}}dxdy\notag\\
&&+\int_{\Omega}\int_{\Omega^c}\frac{|\phi(x)-\phi(y)|^2}{|x-y|^{n+2s}}dxdy\notag\\
&=&\int_{\Omega}\int_{\Omega}\frac{|\phi(x)-\phi(y)|^2}{|x-y|^{n+2s}}dxdy+2\int_{\Omega^c}\int_{\Omega}\frac{|\phi(x)-\phi(y)|^2}{|x-y|^{n+2s}}dxdy.
\end{eqnarray}
Calculate
\begin{eqnarray}\label{e:integrals-2}
&&\int_{\Omega}\int_{\Omega}\frac{|\phi(x)-\phi(y)|^2}{|x-y|^{n+2s}}dxdy\notag\\
&=&\int_{B_{\varepsilon}}\int_{B_{\varepsilon}}\frac{\left|\frac{1}{\varepsilon^{n}}\left(1-\frac{|x|}{\varepsilon}\right)
-\frac{1}{\varepsilon^{n}}\left(1-\frac{|y|}{\varepsilon}\right)\right|^2}{|x-y|^{n+2s}}dxdy
+\int_{\Omega\setminus B_{\varepsilon}}\int_{B_{\varepsilon}}\frac{\left|\frac{1}{\varepsilon^{n}}\left(1-\frac{|x|}{\varepsilon}\right)\right|^2}{|x-y|^{n+2s}}dxdy\notag\\
&&+\int_{B_{\varepsilon}}\int_{\Omega\setminus B_{\varepsilon}}\frac{\left|\frac{1}{\varepsilon^{n}}\left(1-\frac{|y|}{\varepsilon}\right)\right|^2}{|x-y|^{n+2s}}dxdy\notag\\
&=&\int_{B_{\varepsilon}}\int_{B_{\varepsilon}}\frac{\left|\frac{1}{\varepsilon^{n}}\left(1-\frac{|x|}{\varepsilon}\right)
-\frac{1}{\varepsilon^{n}}\left(1-\frac{|y|}{\varepsilon}\right)\right|^2}{|x-y|^{n+2s}}dxdy
+2\int_{\Omega\setminus B_{\varepsilon}}\int_{B_{\varepsilon}}\frac{\left|\frac{1}{\varepsilon^{n}}\left(1-\frac{|x|}{\varepsilon}\right)\right|^2}{|x-y|^{n+2s}}dxdy.\notag\\
&:=&T_1+T_2.
\end{eqnarray}
Estimating $T_1$, we have that
\begin{eqnarray*}
T_1&=&\frac{1}{\varepsilon^{2n+2}}\int_{B_{\varepsilon}}\int_{B_{\varepsilon}}\frac{\left||x|-|y|\right|^2}{|x-y|^{n+2s}}dxdy\notag\\
&\le&\frac{1}{\varepsilon^{2n+2}}\int_{B_{\varepsilon}}\int_{B_{\varepsilon}}\frac{|x-y|^2}{|x-y|^{n+2s}}dxdy
=\frac{1}{\varepsilon^{2n+2}}\int_{B_{\varepsilon}}\int_{B_{\varepsilon}}\frac{1}{|x-y|^{n+2s-2}}dxdy\notag\\
&\le&\frac{1}{\varepsilon^{2n+2}}\int_{B_{\varepsilon}}\int_{B_{2\varepsilon}(y)}\frac{1}{|x-y|^{n+2s-2}}dxdy
=\frac{\omega_{n-1}}{\varepsilon^{2n+2}}\int_{B_{\varepsilon}}\left\{\int_{0}^{2\varepsilon}\frac{1}{r^{n+2s-2}}r^{n-1}dr\right\}dy\notag\\
&\le&\frac{C}{\varepsilon^{n+2s}}.
\end{eqnarray*}
And calculate $T_2$:
\begin{eqnarray*}
T_2&=&\frac{2}{\varepsilon^{2n+2}}\int_{\Omega\setminus B_{\varepsilon}}\int_{B_{\varepsilon}}\frac{\left(\varepsilon-|x|\right)^2}{|x-y|^{n+2s}}dxdy\\
&=&\frac{2}{\varepsilon^{2n+2}}\int_{\Omega\setminus B_{2\varepsilon}}\int_{B_{\varepsilon}}\frac{\left(\varepsilon-|x|\right)^2}{|x-y|^{n+2s}}dxdy+\frac{2}{\varepsilon^{2n+2}}\int_{B_{2\varepsilon}\setminus B_{\varepsilon}}\int_{B_{\varepsilon}}\frac{\left(\varepsilon-|x|\right)^2}{|x-y|^{n+2s}}dxdy.\notag
\end{eqnarray*}
Then estimate
\begin{eqnarray*}
&&\frac{2}{\varepsilon^{2n+2}}\int_{\Omega\setminus B_{2\varepsilon}}\int_{B_{\varepsilon}}\frac{\left(\varepsilon-|x|\right)^2}{|x-y|^{n+2s}}dxdy\notag\\
&=&\frac{2}{\varepsilon^{2n+2}}\int_{B_{\varepsilon}}\int_{\Omega\setminus B_{2\varepsilon}}\frac{\left(\varepsilon-|x|\right)^2}{|x-y|^{n+2s}}dydx
\le\frac{2}{\varepsilon^{2n+2}}\int_{B_{\varepsilon}}\int_{\Omega\setminus B_{\varepsilon}}\frac{\varepsilon^2}{|y|^{n+2s}}dydx\notag\\
&\le&\frac{2\omega_{n-1}}{\varepsilon^{2n}}\int_{B_{\varepsilon}}\left\{\int_{\varepsilon}^{+\infty} \frac{1}{r^{n+2s}}r^{n-1}dr\right\}dx=\frac{\omega_{n-1}}{s\varepsilon^{2n+2s}}\int_{B_{\varepsilon}}dx\notag\\
&\le&\frac{C}{\varepsilon^{n+2s}},
\end{eqnarray*}
and
\begin{eqnarray*}
&&\frac{2}{\varepsilon^{2n+2}}\int_{B_{2\varepsilon}\setminus B_{\varepsilon}}\int_{B_{\varepsilon}}\frac{\left(\varepsilon-|x|\right)^2}{|x-y|^{n+2s}}dxdy\notag\\
&\le&\frac{2}{\varepsilon^{2n+2}}\int_{B_{2\varepsilon}\setminus B_{\varepsilon}}\int_{B_{\varepsilon}}\frac{\left(|y|-|x|\right)^2}{|x-y|^{n+2s}}dxdy
\le\frac{2}{\varepsilon^{2n+2}}\int_{B_{2\varepsilon}\setminus B_{\varepsilon}}\int_{B_{\varepsilon}}\frac{1}{|x-y|^{n+2s-2}}dxdy\notag\\
&\le&\frac{2}{\varepsilon^{2n+2}}\int_{B_{2\varepsilon}\setminus B_{\varepsilon}}\int_{B_{3\varepsilon}(y)}\frac{1}{|x-y|^{n+2s-2}}dxdy
=\frac{2\omega_{n-1}}{\varepsilon^{2n+2}}\int_{B_{2\varepsilon}\setminus B_{\varepsilon}}\left\{\int_{0}^{3\varepsilon}\frac{1}{r^{n+2s-2}}r^{n-1}dr\right\}dy\notag\\
&=&\frac{2\cdot 3^{2-2s}\omega_{n-1}}{\varepsilon^{2n+2s}}\int_{B_{2\varepsilon}\setminus B_{\varepsilon}}dx
=\frac{2\cdot 3^{2-2s}(2^n-1)\omega_{n-1}^2}{n\varepsilon^{n+2s}}.
\end{eqnarray*}
Therefore, we obtain that
\begin{eqnarray}\label{e:integrals-1-2}
\int_{\Omega}\int_{\Omega}\frac{|\phi(x)-\phi(y)|^2}{|x-y|^{n+2s}}dxdy=T_1+T_2\le \frac{C}{\varepsilon^{n+2s}},
\end{eqnarray}
where $C$ is a positive constant depending on $n,\, s,\,\Omega$.

Finally, from $B_{2\varepsilon}\subset \Omega$, we have
\begin{eqnarray}\label{e:integral-3}
&&2\int_{\Omega^c}\int_{\Omega}\frac{|\phi(x)-\phi(y)|^2}{|x-y|^{n+2s}}dxdy\notag\\
&=&2\int_{\Omega^c}\int_{B_{\varepsilon}}\frac{\frac{1}{\varepsilon^{2n}}\left(1-\frac{|x|}{\varepsilon}\right)^2}{|x-y|^{n+2s}}dxdy
\le \frac{C}{\varepsilon^{2n}}\int_{\Omega^c}\int_{B_{\varepsilon}}\frac{\left(1-\frac{|x|}{\varepsilon}\right)^2}{\left|\frac{y}{2}\right|^{n+2s}}dxdy\notag\\
&\le&\frac{C}{\varepsilon^{2n}}\int_{B_{2\varepsilon}^c}\int_{B_{\varepsilon}}\frac{\left(1-\frac{|x|}{\varepsilon}\right)^2}{\left|\frac{y}{2}\right|^{n+2s}}dxdy
\le \frac{C}{\varepsilon^{n+2s}}.
\end{eqnarray}
Summarizing the estimates (\ref{e:phi-integral}), (\ref{e:integrals-1}), (\ref{e:integrals-2}), (\ref{e:integrals-1-2}) and (\ref{e:integral-3}), we have that
\begin{equation}\label{e:phi-3}
\frac{C_{n,s}\varepsilon^{2s}}{2}\int_{\mathbf R^{2n}\setminus(\Omega^c)^2}\frac{|\phi(x)-\phi(y)|^2}{|x-y|^{n+2s}}dxdy\le \frac{C}{\varepsilon^{n}},
\end{equation}
where $C$ is a positive constant depending on $n,s,\Omega$.
From (\ref{e:phi-2}) and (\ref{e:phi-3}), we obtain (\ref{e:phi-estimate}). This completes the proof.
\end{proof}
Define
\begin{eqnarray*}
g(t)=I_{\varepsilon}(t\phi),\quad t\ge 0.
\end{eqnarray*}
\begin{lemma}\label{l:t-large}
Assume that $\varepsilon>0$ sufficiently small, there exist $t_1$ and $t_2$ with $0<t_1<t_2$ such that
\begin{enumerate}
  \item[(1)] for $t>t_1$, $g'(t)<0$;
  \item[(2)] for $t>t_2$, $g(t)<0$.
\end{enumerate}

\end{lemma}
\begin{proof}
By a similar argument as in the proof of Lemma \ref{l:phi-estimate}, we obtain that, for $t>0$,
\begin{eqnarray*}
\frac{C_{n,s}\varepsilon^{2s}}{4}\int_{\mathbf R^{2n}\setminus(\Omega^c)^2}\frac{|t\phi(x)-t\phi(y)|^2}{|x-y|^{n+2s}}dxdy+\frac{1}{2}\int_{\Omega}(t\phi)^2dx
\le C_0t^2\varepsilon^{-n},
\end{eqnarray*}
where $C_0$ is a positive constant depending on $n,s, \Omega$.
Moreover,
\begin{eqnarray*}
\int_{\Omega}(t\phi)_+^{p+1}dx
&=&\frac{t^{p+1}}{\varepsilon^{n(p+1)}}\int_{B_{\varepsilon}}\left(1-\frac{|x|}{\varepsilon}\right)^{p+1}dx\notag\\
&=&\frac{\omega_{n-1} t^{p+1}}{\varepsilon^{n(p+1)}}\int_0^{\varepsilon}\left(1-\frac{r}{\varepsilon}\right)^{p+1}r^{n-1}dr\notag\\
&=&\frac{\omega_{n-1} t^{p+1}}{\varepsilon^{np}}\int_0^{1}\left(1-\rho\right)^{p+1}\rho^{n-1}d\rho=\frac{\alpha\omega_{n-1}t^{p+1}}{\varepsilon^{np}},
\end{eqnarray*}
where $\alpha:=\int_0^{1}\left(1-\rho\right)^{p+1}\rho^{n-1}d\rho$. Let $t_2=\left(\frac{C_0(p+1)}{\alpha\omega_{n-1}}\right)^{\frac{1}{p-1}}\varepsilon^{n}$. Then for all $t>t_2$, it holds that $g(t)=I_{\varepsilon}(t\phi)<0$.

Next compute
\begin{eqnarray*}
g'(t)&=&\frac{C_{n,s}\varepsilon^{2s}t}{2}\int_{\mathbf R^{2n}\setminus(\Omega^c)^2}\frac{|\phi(x)-\phi(y)|^2}{|x-y|^{n+2s}}dxdy+t\int_{\Omega}\phi^2dx
-t^p\int_{\Omega}\phi^{p+1}dx\notag\\
&\le& 2C_0t\varepsilon^{-n}-\frac{\alpha\omega_{n-1}t^{p}}{\varepsilon^{np}}.
\end{eqnarray*}
Let $t_1=\left(\frac{2C_0}{\alpha\omega_{n-1}}\right)^{\frac{1}{p-1}}\varepsilon^{n}$. Thus if choosing $t>t_1$, we have that $g'(t)<0$.
Since $p>1$, it holds that $t_1<t_2$. This completes the proof.
\end{proof}

\begin{proof}[Proof of Lemma \ref{l:energy-estimate}]
By Lemma \ref{l:positive}, it holds that $g(t)>0$ if $t$ is positive and sufficiently small. Then from Lemma \ref{l:t-large}, we have that there exists a $t_0>0$ such that $g(t_0)=0$. Moreover, estimate
\begin{eqnarray*}
\max_{t\ge 0} g(t)&=&\max_{0\le t\le t_1}g(t)\\
&\le & \max_{0\le t\le t_1}\left\{ C_0t^2\varepsilon^{-n}-\frac{1}{p+1}\int_{\Omega}(t\phi)^{p+1}dx\right\}\\
&\le &\max_{0\le t\le t_1}C_0t^2\varepsilon^{-n}\\
&=& C_0t_1^2\varepsilon^{-n}=C_1\varepsilon^{n},
\end{eqnarray*}
where $C_1=C_0\left(\frac{2C_0}{\alpha\omega_{n-1}}\right)^{\frac{2}{p-1}}$. It completes the proof.
\end{proof}

\begin{proof}[Proof of Theorem \ref{t:np}]
(1) Let $e=t_0\phi$ and $\Gamma=\{\gamma\in C([0,1],H^s_{\varepsilon,\Omega})\,|\,\gamma(0)=0 \mbox{ and }\gamma(1)=e\}$. Then by Mountain Pass Theorem, we have that $$c_{\varepsilon}:=\inf_{\gamma\in \Gamma}\sup_{s\in [0,1]} I_{\varepsilon}(\gamma(s))>0$$
is a critical value of $I_{\varepsilon}$. Then there exists a critical point $u_{\varepsilon}$ such that
\begin{eqnarray*}
I_{\varepsilon}(u_{\varepsilon})=c_{\varepsilon}\le C_1\varepsilon^{n}.
\end{eqnarray*}

Note that the unique constant solution to (\ref{e:np}) is $u\equiv 1$ on $\mathbf R^n$. A direct calculate yields
\begin{eqnarray*}
I_{\varepsilon}(1)=\left(\frac{1}{2}-\frac{1}{p+1}\right)|\Omega|>0,
\end{eqnarray*}
where $|\Omega|$ denotes the volume of $\Omega$. Thus
for $\varepsilon $ small enough, we have that
\begin{equation*}
I_{\varepsilon}(u_{\varepsilon})<I_{\varepsilon}(1).
\end{equation*}
Therefore, $u_{\varepsilon}$ is a nonconstant solution to (\ref{e:np}).

(2) Since $u_{\varepsilon}$ is a solution to (\ref{e:np}), we have that
\begin{equation}\label{e:critical-e}
\frac{C_{n,s}\varepsilon^{2s}}{2}\int_{\mathbf R^{2n}\setminus(\Omega^c)^2}\frac{|u_{\varepsilon}(x)-u_{\varepsilon}(y)|^2}{|x-y|^{n+2s}}dxdy
+\int_{\Omega}u_{\varepsilon}^2dx=\int_{\Omega}|u_{\varepsilon}|^{p+1}dx.
\end{equation}
Then by the definition of $I_{\varepsilon}$, it holds that
\begin{eqnarray*}
&&I_{\varepsilon}(u_{\varepsilon})\\
&=&\frac{1}{2}\left(\frac{C_{n,s}\varepsilon^{2s}}{2}\int_{\mathbf R^{2n}\setminus(\Omega^c)^2}\frac{|u_{\varepsilon}(x)-u_{\varepsilon}(y)|^2}{|x-y|^{n+2s}}dxdy+\int_{\Omega}u_{\varepsilon}^2dx\right)
-\frac{1}{p+1}\int_{\Omega}|u_{\varepsilon}|^{p+1}dx\\
&=& \left(\frac{1}{2}-\frac{1}{p+1}\right)\left(\frac{C_{n,s}\varepsilon^{2s}}{2}\int_{\mathbf R^{2n}\setminus(\Omega^c)^2}\frac{|u_{\varepsilon}(x)-u_{\varepsilon}(y)|^2}{|x-y|^{n+2s}}dxdy+\int_{\Omega}u_{\varepsilon}^2dx\right)\\
&=&\left(\frac{1}{2}-\frac{1}{p+1}\right)\|u_{\varepsilon}\|_{H^s_{\varepsilon,\Omega}}^2.
\end{eqnarray*}
Thus, choosing $C_{2}=2C_1(\frac{p+1}{p-1})$, we obtain (\ref{e:solution-norm}).

(3) We prove that we there exists a critical point $u_{\varepsilon}\ge 0$ in $\mathbf R^n$. In fact, when $u_{\varepsilon}\le 0$, we can choose $-u_{\varepsilon}$. Thus we only need to exclude the sign change case. We argue by contradiction. Assume that $u_{\varepsilon}$ is a sign change critical point obtained above. Note that
\begin{equation*}
(|u_{\varepsilon}(x)|-|u_{\varepsilon}(y)|)^2\le |u_{\varepsilon}(x)-u_{\varepsilon}(y)|^2.
\end{equation*}
Then we have the following Kato-type inequality
\begin{eqnarray}\label{e:e:kti}
\int_{\mathbf R^{2n}\setminus(\Omega^c)^2}\frac{(|u_{\varepsilon}(x)|-|u_{\varepsilon}(y)|)^2}{|x-y|^{n+2s}}dxdy<\int_{\mathbf R^{2n}\setminus(\Omega^c)^2}\frac{|u_{\varepsilon}(x)-u_{\varepsilon}(y)|^2}{|x-y|^{n+2s}}dxdy.
\end{eqnarray}
The strict inequality is from that $u_{\varepsilon}$ is sign change.
It follows that
\begin{equation}\label{e:fkti}
I_{\varepsilon}(|u_{\varepsilon}|)< I_{\varepsilon}(u_{\varepsilon}).
\end{equation}

Let
\begin{eqnarray*}
&&f(t)=I_{\varepsilon}(t|u_{\varepsilon}|)\\
&=&\frac{t^2}{2}\left(\frac{C_{n,s}\varepsilon^{2s}}{2}\int_{\mathbf R^{2n}\setminus(\Omega^c)^2}\frac{(|u_{\varepsilon}(x)|-|u_{\varepsilon}(y)|)^2}{|x-y|^{n+2s}}dxdy
+\int_{\Omega}u_{\varepsilon}^2dx\right)-\frac{t^{p+1}}{p+1}\int_{\Omega}|u_{\varepsilon}|^{p+1}dx,\notag
\end{eqnarray*}
where $t\in [0,+\infty)$. For simplicity, set
$$\Lambda:=\left(\frac{C_{n,s}\varepsilon^{2s}}{2}\int_{\mathbf R^{2n}\setminus(\Omega^c)^2}\frac{|u_{\varepsilon}(x)-u_{\varepsilon}(y)|^2}{|x-y|^{n+2s}}dxdy
+\int_{\Omega}u_{\varepsilon}^2dx\right)>0,$$
$$\bar\Lambda:=\left(\frac{C_{n,s}\varepsilon^{2s}}{2}\int_{\mathbf R^{2n}\setminus(\Omega^c)^2}\frac{(|u_{\varepsilon}(x)|-|u_{\varepsilon}(y)|)^2}{|x-y|^{n+2s}}dxdy
+\int_{\Omega}u_{\varepsilon}^2dx\right)>0,$$
and
$$\Xi:=\int_{\Omega}|u_{\varepsilon}|^{p+1}dx>0.$$
Thus
\begin{eqnarray*}
f(t)=\frac{\bar\Lambda}{2}t^2-\frac{\Xi}{p+1}t^{p+1},
\end{eqnarray*}
and
\begin{eqnarray*}
\bar\Lambda<\Lambda.
\end{eqnarray*}
Since $u_{\varepsilon}$ is critical point, (\ref{e:critical-e}) yields
\begin{eqnarray*}
\Lambda=\Xi\quad\mbox{and}\quad\frac{\Lambda}{2}-\frac{\Xi}{p+1}=c_{\varepsilon}.
\end{eqnarray*}
Note that $f(t)\to -\infty$ as $t\to +\infty$. $f$ has a unique maximum point at $t_1=(\frac{\bar\Lambda}{\Xi})^{\frac{1}{p-1}}$ and
\begin{eqnarray*}
f(t_1)&=&\frac{\bar\Lambda}{2}\left(\frac{\bar\Lambda}{\Xi}\right)^{\frac{2}{p-1}}
-\frac{\Xi}{p+1}\left(\frac{\bar\Lambda}{\Xi}\right)^{\frac{p+1}{p-1}}\notag\\
&=&\left(\frac{\bar\Lambda}{\Xi}\right)^{\frac{2}{p-1}}\left(\frac{\bar\Lambda}{2}
+\left(c-\frac{\Lambda}{2}\right)\left(\frac{\bar\Lambda}{\Xi}\right)\right)\notag\\
&=&\left(\frac{\bar\Lambda}{\Xi}\right)^{\frac{2}{p-1}}c_{\varepsilon}<c_{\varepsilon}.
\end{eqnarray*}
Thus, there exists $t_2>0$ such that $f(t_2)=0$. Let $\bar e=t_2u_{\varepsilon}$. We now set $V^+=\left\{\lambda e+\mu \bar e\,|\,\lambda\ge0,\,\mu\ge 0\right\}$. Therefore, there is an $R_0>\max\{\|e\|_{H^s_{\varepsilon,\Omega}},\,\|\bar e\|_{H^s_{\varepsilon,\Omega}}\}$ such that for all $u\in V^+$ with $\|u\|_{H^s_{\varepsilon,\Omega}}\ge R_0$, it holds that $I_{\varepsilon}(u)<0$. Let $\gamma_0$ be the path which consists of the line segment with endpoints $0$ and $R_0\bar e/\|\bar e\|_{H^s_{\varepsilon,\Omega}}$, the circular arc $\partial B_{R_0}\cap V^+$, and the line segment with endpoints $R_0 e/\|e\|_{H^s_{\varepsilon,\Omega}}$ and $e$. Hence $\gamma_0$ belongs to $\Gamma$. However, along $\gamma_0$, $I_{\varepsilon}$ is positive only on the line joining $0$ and $\bar e$. This yields that
\begin{eqnarray*}
\sup_{u\in \gamma_0}I_{\varepsilon}(u)=f(t_1)<c_{\varepsilon}.
\end{eqnarray*}
It is a contradiction. Therefore, there is a critical point $u_{\varepsilon}\ge 0$.
This completes the proof.
\end{proof}


\section{Proof of Theorem \ref{t:integrability}}\label{s:properties}

\begin{proof}[Proof of Theorem \ref{t:integrability}]
(1) Without loss of generality, we assume that $0\in \Omega$. Then there exists $R_0>0$ such that $B_{2R_0}\subset \Omega$.
Since $u\in H^s_{\varepsilon,\Omega}$, it holds that
\begin{equation*}
I:=\int_{\Omega}\int_{\Omega^c}\frac{|u(x)-u(y)|^2}{|x-y|^{n+2s}}dydx<+\infty.
\end{equation*}
Particularly,
\begin{equation*}
\int_{B_{R_0}}\int_{\Omega_{\rho}\cap\Omega^c}\frac{|u(x)-u(y)|^2}{|x-y|^{n+2s}}dydx\le I<+\infty.
\end{equation*}
Here $\rho>0$ is constant and $\Omega_{\rho}=\{y\in\mathbf R^n\,|\,d(y,\Omega)<\rho\}$.
A direct computation yields
\begin{eqnarray*}
&&\int_{B_{R_0}}\int_{\Omega_{\rho}\cap\Omega^c}\frac{|u(x)-u(y)|^2}{|x-y|^{n+2s}}dydx\notag\\
&\ge&\int_{B_{R_0}}\int_{\Omega_{\rho}\cap\Omega^c}\frac{|u(x)|^2}{|x-y|^{n+2s}}dydx
+\int_{B_{R_0}}\int_{\Omega_{\rho}\cap\Omega^c}\frac{|u(y)|^2}{|x-y|^{n+2s}}dydx\notag\\
&&-2\int_{B_{R_0}}\int_{\Omega_{\rho}\cap\Omega^c}\frac{|u(x)u(y)|}{|x-y|^{n+2s}}dydx:=T_1+T_2-T_3.
\end{eqnarray*}
We now estimate these three terms. Firstly,
\begin{eqnarray*}
T_1&=&\int_{B_{R_0}}|u(x)|^2\left\{\int_{\Omega_{\rho}\cap\Omega^c}\frac{1}{|x-y|^{n+2s}}dy\right\}dx\notag\\
&\ge& \frac{|\Omega_{\rho}\cap\Omega^c|}{(d(\Omega)+\rho)^{n+2s}}\int_{B_{R_0}}|u(x)|^2dx:=a.
\end{eqnarray*}
where $d(\Omega)$ denotes the diameter of $\Omega$ and $|\Omega_{\rho}\cap\Omega^c|$ is the volume of $\Omega_{\rho}\cap\Omega^c$. Note that $\int_{B_{R_0}}|u(x)|^2dx<\infty$
since $\|u\|_{H^s(\Omega)}\le\|u\|_{H^s_{\varepsilon,\Omega}}$. Thus $a$ is a nonnegative constant depending on $\Omega, R_0,\rho, n,s$. Secondly,
\begin{eqnarray*}
T_2&\ge& \frac{1}{(d(\Omega)+\rho)^{n+2s}}\int_{B_{R_0}}\int_{\Omega_{\rho}\cap\Omega^c}|u(y)|^2dydx\notag\\
&=& \frac{|B_{R_0}|}{(d(\Omega)+\rho)^{n+2s}}\int_{\Omega_{\rho}\cap\Omega^c}|u(y)|^2dy\notag\\
&:=&b\int_{\Omega_{\rho}\cap\Omega^c}|u(y)|^2dy.
\end{eqnarray*}
Here $b$ is a positive constant depending on $\Omega, R_0,\rho, n,s$.
Finally,
\begin{eqnarray*}
T_3&=&2\int_{B_{R_0}}|u(x)|\int_{\Omega_{\rho}\cap\Omega^c}\frac{|u(y)|}{|x-y|^{n+2s}}dydx\notag\\
&\le& \frac{2}{d(B_{R_0},\partial \Omega)^{n+2s}}\int_{B_{R_0}}|u(x)|\left\{\int_{\Omega_{\rho}\cap\Omega^c}|u(y)|dy\right\}dx\notag\\
&=& \frac{2\int_{B_{R_0}}|u(x)|dx}{d(B_{R_0},\partial \Omega)^{n+2s}}\int_{\Omega_{\rho}\cap\Omega^c}|u(y)|dy\notag\\
&:=&c\int_{\Omega_{\rho}\cap\Omega^c}|u(y)|dy,
\end{eqnarray*}
where $c$ is also a nonnegative constant depending on $\Omega, R_0,\rho, n,s$. Therefore, we obtain that
\begin{eqnarray}\label{e:ior}
I\ge a+b\int_{\Omega_{\rho}\cap\Omega^c}|u(y)|^2dy-c\int_{\Omega_{\rho}\cap\Omega^c}|u(y)|dy.
\end{eqnarray}
Note that by the proof of (\ref{e:ior}), it holds that, for any $X\subset \Omega_{\rho}\cap\Omega^c$,
\begin{equation*}
I\ge a+b\int_{X}|u(y)|^2dy-c\int_{X}|u(y)|dy.
\end{equation*}
We then argue by contradiction. Assume that $u\notin L^2(\Omega_{\rho}\cap\Omega^c)$, that is
$$\int_{\Omega_{\rho}\cap\Omega^c}|u(y)|^2dy=+\infty.$$
From (\ref{e:ior}), we have that $u\notin L^1(\Omega_{\rho}\cap\Omega^c)$. Let
\begin{equation*}
A_k:=\{y\in \Omega_{\rho}\cap\Omega^c\,|\,|u(y)|>2^k\}
\end{equation*}
and
\begin{equation*}
D_k:=A_k\setminus A_{k+1}=\{y\in \Omega_{\rho}\cap\Omega^c\,|\,2^k<|u(y)|\le 2^{k+1}\}.
\end{equation*}
Set $d_k$ to be the measure of $D_k$. Let $N_1$ be a positive integer such that $2^{N_1-1}>\frac{c}{b}$. Then, for all $N_2>N_1$,
\begin{eqnarray}\label{e:an1}
\int_{A_{N_1}\setminus A_{N_2}}|u(y)|dy=\sum_{k=N_1}^{N_2}\int_{A_{k}\setminus A_{k+1}}|u(y)|dy\le \sum_{k=N_1}^{N_2}2^{k+1}d_k\to +\infty,\quad \mbox{as }N_2\to \infty,
\end{eqnarray}
and
\begin{eqnarray*}
\int_{A_{N_1}\setminus A_{N_2}}|u(y)|^2dy=\sum_{k=N_1}^{N_2}\int_{A_{k}\setminus A_{k+1}}|u(y)|^2dy\ge \sum_{k=N_1}^{N_2}2^{2k}d_k\to +\infty,\quad \mbox{as }N_2\to \infty.
\end{eqnarray*}
Since
\begin{equation*}
\sum_{k=N_1}^{N_2}2^{2k}d_k> 2^{N_1-1}\sum_{k=N_1}^{N_2}2^{k+1}d_k,
\end{equation*}
we have that
\begin{eqnarray*}
I&\ge& a+b\int_{A_{N_1}\setminus A_{N_2}}|u(y)|^2dy-c\int_{A_{N_1}\setminus A_{N_2}}|u(y)|dy\notag\\
&>&a+2^{N_1-1}b\sum_{k=N_1}^{N_2}2^{k+1}d_k-c\sum_{k=N_1}^{N_2}2^{k+1}d_k.
\end{eqnarray*}
It is a contradiction to (\ref{e:an1}). Therefore, $u\in L^2(\Omega_{\rho}\cap\Omega^c)$. Note that $u\in L^2(\Omega)$, we obtain that
$u\in L^2(\Omega_{\rho})$. Since $\rho$ is arbitrary, it follows that $u\in L^2_{{\rm loc}}(\mathbf R^n)$.

(2) Let $R$ be a positive constant such that $\Omega\subset B_R(0)$. By the proof of Proposition 3.13 in \cite{DRV:NPNBC}, we know that if $\mathcal N_s(u)=0$, then $u$ is bounded in $B_R^c(0)$ with
\begin{equation*}
\lim_{|x|\to \infty}u(x)=\frac{1}{|\Omega|}\int_{\Omega}u(x)dx \quad \mbox{uniformly in }x.
\end{equation*}
Therefore,
\begin{equation*}
\int_{B_{R}^c(0)}\frac{|u(x)|}{1+|x|^{n+2s}}dx\le C\int_{B_{R}^c(0)}\frac{1}{1+|x|^{n+2s}}dx<+\infty,
\end{equation*}
where $C$ is a constant such that $\sup_{B_R^c(0)}|u|\le C$. So,
we only need to consider $u$ on $B_{R}(0)$. From the conclusion (1), it follows that
\begin{equation*}
u|_{B_R(0)}\in L^2(B_R(0)).
\end{equation*}
Thus, we have that
\begin{eqnarray*}
\int_{B_R(0)}\frac{|u(x)|}{1+|x|^{n+2s}}dx<+\infty.
\end{eqnarray*}
This completes the proof.
\end{proof}



\newcommand{\Toappear}{to appear in}

\bibliography{mrabbrev,cz_abbr2003-0,localbib}

\bibliographystyle{plain}
\end{document}